\documentclass[a4paper]{amsart}
\usepackage[T1]{fontenc}
\usepackage{lmodern}
\usepackage[english]{babel}
\usepackage{amsmath}
\usepackage{amssymb}
\usepackage{amsthm}
\usepackage{xcolor}
\usepackage[pdfauthor={},
			pdftitle={},
			pdfsubject={},
			pdfkeywords={},
			pdfencoding=auto
			]{hyperref}
\usepackage[capitalise]{cleveref}
\usepackage{algorithm}
\usepackage{algpseudocode}

\newtheorem{theorem}{Theorem}[section]
\newtheorem{lemma}[theorem]{Lemma}
\newtheorem{corollary}[theorem]{Corollary}

\newtheorem{proposition}[theorem]{Proposition}
\theoremstyle{definition}
\newtheorem{definition}[theorem]{Definition}
\newtheorem{example}[theorem]{Example}
\newtheorem{remark}[theorem]{Remark}
\numberwithin{equation}{section}
\crefname{equation}{}{}

\DeclareMathOperator{\ord}{ord}
\DeclareMathOperator{\lcm}{lcm}
\DeclareMathOperator{\img}{im}

\DeclareMathOperator{\rank}{rank}
\DeclareMathOperator{\bigO}{\mathcal{O}}

\newcommand{\NN}{\mathbb{N}}
\newcommand{\ZZ}{\mathbb{Z}}
\newcommand{\FF}{\mathbb{F}}
\newcommand{\FFC}{\overline{\mathbb{F}}}
\newcommand{\FFA}{\mathcal{F}}
\newcommand{\Ga}{\mathcal A}
\newcommand{\Gb}{\mathcal B}
\newcommand{\Gap}{\Ga^{\textnormal{p}}}
\newcommand{\Gbp}{\Gb^{\textnormal{p}}}

\title[Factorization and irreducibility of composed products]{Factorization and irreducibility of composed products}
\subjclass{11T06, 12E20}
\date{\today}
\author{Lukas Kölsch}
\author{Lucas Krompholz}
\author{Gohar Kyureghyan}
\address{L. Kölsch, University of South Florida, USA}
\email{lukas.koelsch.math@gmail.com}
\address{L. Krompholz, University of Rostock, Germany}
\email{lucas.krompholz@uni-rostock.de}
\address{G. Kyureghyan, University of Rostock, Germany}
\email{gohar.kyureghyan@uni-rostock.de}

\begin{document}
\begin{abstract}
Brawley and Carlitz introduced diamond products of elements of finite fields and associated composed products of polynomials in 1987.
Composed products yield a method to construct irreducible polynomials of  large composite degrees from irreducible polynomials of lower degrees. We show that the composed product of two irreducible polynomials of degrees $m$ and $n$ is again irreducible if and only if $m$ and $n$ are coprime and the involved diamond product satisfies a special cancellation property, the so-called conjugate cancellation.  This completes the characterization of irreducible composed products, considered in several previous papers.
More generally, we give precise criteria when a diamond product satisfies conjugate cancellation.
For diamond products defined via bivariate polynomials, we prove simple criteria that characterize when conjugate cancellation holds. We also provide efficient algorithms to check these criteria. We achieve stronger results as well as more efficient algorithms in the case that the polynomials are bilinear. Lastly, we consider possible constructions of normal elements using composed products and the methods we developed.
\end{abstract}
\maketitle

\section{Introduction}

Polynomials over the finite field $\FF_q$ are fundamental objects in theoretical aspects and in variety of applications
of finite fields. Explicit constructions of irreducible polynomials and the study of factorization of
large families of polynomials are research problems with high impact on the theory and applications of finite fields, see~\cite[Chapter 3]{mullen2013handbook} for a concise overview.

Brawley and Carlitz \cite{brawley1987IrreduciblesComposedProduct} introduced composed products as a root-based method to construct irreducible polynomials over $\FF_q$ of large degrees from irreducible polynomials of small degrees. The construction uses so-called \emph{diamond products}, which are binary operations $\diamond:G\times G\rightarrow G$, satisfying
\[ (\alpha\diamond\beta)^q = \alpha^q\diamond\beta^q \]
for all $\alpha,\beta\in G$. Here, $G$ is a subset of an algebraic closure $\FFC_q$ which is invariant under the Frobenius automorphism $\sigma:x\rightarrow x^q$, i.e., $\alpha^q\in G$ for all $\alpha\in G$.
Given a diamond product, the composed product\footnote{Note the difference from the usual notion of composition, i.e., $(f\circ g)(x) = f(g(x))$.} of two monic polynomials $f,g\in\FF_q[X]$ is defined as
\[ f\diamond g = \prod_{\alpha}\prod_{\beta} (X-\alpha\diamond\beta)\]
where $\alpha$ and $\beta$ range over the roots of $f$ and $g$, respectively.

It is shown in \cite{brawley1987IrreduciblesComposedProduct} that if $(G,\diamond)$ forms a group, then the composed product $f\diamond g$ is irreducible if and only if $f$ and $g$ are irreducible and have coprime degrees. Furthermore, it was observed in \cite{brawley1987IrreduciblesComposedProduct} that the same property holds if $(G,\diamond)$ is a semigroup satisfying the cancellation rule. However, all diamond products that form cancellation semigroups are already groups. The authors of \cite{munemasa2016NoteBrawleyCarlitzTheorem} noticed that a weaker hypothesis on the cancellation property already implies the result. Indeed, it is enough when the so-called weak cancellation is satisfied:
\begin{align*}
	\alpha \diamond \beta = \alpha' \diamond \beta &\implies \alpha = \alpha',\\
	\alpha \diamond \beta = \alpha \diamond \beta' &\implies \beta = \beta',
\end{align*}
for all $\alpha,\alpha'\in\FFA_q(m)$ and $\beta,\beta'\in\FFA_q(n)$, where
\begin{equation}\label{eq:finite-field-gens}
	\FFA_q(n) = \{\alpha\in\FF_{q^n} \mid \FF_q(\alpha) = \FF_{q^n}\}
\end{equation}
and $\FFA_q(m)\cup\FFA_q(n)\subseteq G$.
The attentive reader of \cite{munemasa2016NoteBrawleyCarlitzTheorem} notices that in its proofs even a weaker cancellation property than the above one is used. This  was independently mentioned in \cite{irimagzi2023DiamondProductsEnsuring} and \cite{krompholz}. The proofs of
\cite{munemasa2016NoteBrawleyCarlitzTheorem} remain valid if the following cancellation rule is fulfilled.

\begin{definition}[Conjugate Cancellation on $G$]\label{Def:Conjugate-Cancellation-Original}
	Let $n,m$ be integers and $G\subseteq \FFC_q$ be a Frobenius invariant subset with $\FFA_q(m),\FFA_q(n)\subseteq G$. A diamond product $\diamond:G\times G\rightarrow\FFC_q$ is said to satisfy conjugate cancellation on $\FFA_q(m)\times\FFA_q(n)$ if for every integer $k$  it holds that:
	\begin{align*}
		\alpha \diamond \beta = \alpha^{q^k} \diamond \beta &\implies \alpha = \alpha^{q^k}\\
		\alpha \diamond \beta = \alpha \diamond \beta^{q^k} &\implies \beta = \beta^{q^k}
	\end{align*}
	where $\alpha \in \FFA_q(m)$, $\beta \in \FFA_q(n)$.
\end{definition}
In \cite{irimagzi2023DiamondProductsEnsuring} this property was called \emph{weaker cancellation}. We prefer to use, in our opinion, the more self-explanatory term \emph{conjugate cancellation}. In this paper, we extend the definition of conjugate cancellation by requiring that
the implications  hold only for integers $k$ dividing $\gcd(m,n)$, see \cref{Def:Conjugate-Cancellation}. This allows us to cover a larger set of diamond products, as explained in \cref{Sec:Opimality}.\\

{In this paper we study diamond products
satisfying conjugate cancellation and the associated composed products.  In particular in \cref{Thm:General-Brawley-Carlitz}, we show that the composed product of two irreducible polynomials of degrees $m$ and $n$ is irreducible  if and only if $m$ and $n$ are coprime and the involved diamond product satisfies   conjugate cancellation.
More generally, we show that a diamond product of $\alpha \in \FFA_q(m)$ and $\beta \in \FFA_q(n)$ satisfies conjugate cancellation if and only
if every irreducible factor of the composed product
of their minimal polynomials  has a degree $r$ satisfying $\lcm(m,n) = \lcm(m,r) =\lcm(n,r)$.
This is a direct consequence of \cref{th:g_values}
and more details on it are given in \cref{sub_sec:decopmosition}.
In Section~\ref{Sec:Construction}, we give several criteria that ensure that a diamond product defined by a
bivariate polynomial satisfies conjugate
cancellation. We give strong and general criteria in Theorem~\ref{Cor:Conj-Cancel-nice-coeffs}, as well as efficient algorithms that check these criteria in Section~\ref{Sec:Algorithm}. In Section~\ref{Sec:Normal}, we investigate diamond products constructed via linearized polynomials, where we can use normal bases to achieve more general necessary conditions for such a diamond product to satisfy conjugate cancellation (Theorem~\ref{Lma:Coeff-Polys-Linear}). We also provide an adaptation of our previous algorithms for this case. Lastly, we consider when diamond products can be used to construct normal elements, and give some precise results for specific diamond products (Theorem~\ref{thm:albert}).}\\

\section{Factorization of composed products of polynomials}\label{Sec:Opimality}

Let $\Ga \subseteq\FFC_q$ be an arbitrary Frobenius invariant set. For  $m\in\NN$ we define
\[\Ga_m = \Ga \cap \FF_{q^m}\qquad \text{and} \qquad \Gap_m = \Ga\cap\FFA_q(m).\]
Note that $\Ga_m$ is finite and $\Ga_m\cap \Ga_n = \Ga_{\gcd(m,n)}$.
The set $\FFA_q(m)$, defined in \cref{eq:finite-field-gens}, consists of all elements of $\FF_{q^m}$ which do not belong to any proper subfield of the form $\FF_{q^k}$ with $k < m$. Therefore, the elements of $\FFA_q(m)$, and hence those from $\Gap_m$, have a minimal polynomial of degree $m$ over $\FF_q$.
Recall that the minimal polynomial $f$ of $\alpha \in \FFA_q(m)$ over $\FF_q$ factorizes as follows in $\FF_{q^m}[X]$
\[f = \prod_{i=0}^{m-1}(X-\alpha^{q^i}).\]

In this paper we consider diamond products $\diamond$ on $\Ga\times\Gb$, where $\Ga,\Gb\subseteq\FFC_q$ are arbitrary Frobenius invariant subsets.  In previous papers mainly the case $\Ga=\Gb=G$ was studied. We call $\diamond:\Ga\times\Gb\rightarrow\FFC_q$ a \emph{diamond product} if
\begin{equation} \label{eq:g_dimond}
	(\alpha\diamond\beta)^q=\alpha^q\diamond\beta^q
\end{equation}
for any $(\alpha,\beta)\in\Ga\times\Gb$. Note that $\alpha \diamond \beta \in \FF_{q^{\lcm(m,n)}}$
if $\alpha \in \Gap_m$ and $\beta \in \Gbp_n$.
Given $\alpha\in\Gap_m$ and $\beta \in \Gbp_n$ with minimal polynomials
\[f = \prod_{i=0}^{m-1}(X-\alpha^{q^i}) \text{ and } g= \prod_{j=0}^{n-1}(X-\beta^{q^j}),\]
resp., the composed product of the polynomials $f$ and $g$ is then defined as
\begin{equation} \label{eq:g_dppol}
	f\diamond g = \prod_{i=0}^{m-1}\prod_{j=0}^{n-1} (X-\alpha^{q^i}\diamond\beta^{q^j}) \in \FF_q[X].
\end{equation}
Hence, to compute $f\diamond g$ we need the $mn$ values $\alpha^{q^i}\diamond\beta^{q^j}$ of the diamond product.
Next we observe that these $mn$ values can be uniquely determined if we have the values of the diamond product
for $\gcd(m,n)$ many pairs $(\alpha^{q^i},\beta^{q^j})$.
Indeed, by the definition of the diamond product, the value of $\alpha^{q^u} \diamond \beta^{q^v}$ for a fixed pair $(u,v)$
determines the values for all $\alpha^{q^i}\diamond\beta^{q^j}$ for which there is a $t$ satisfying
\begin{equation}\label{eq:g_orbit}
	(\alpha^{q^i})^{q^t} = \alpha^{q^{i+t}}= \alpha^{q^u} \text{ and }  (\beta^{q^j})^{q^t} = \beta^{q^{j+t}}= \beta^{q^v}.
\end{equation}
If \cref{eq:g_orbit} holds, we say that $(\alpha^{q^u}, \beta^{q^v})$ and $(\alpha^{q^i}, \beta^{q^j})$ belong to the same orbit.
Since
$\alpha^{q^m}=\alpha$ and $\beta^{q^n } = \beta$, we may consider, abusing the notation, the exponents $i$ and $j$ as
elements of $\mathbb{Z}_m$ and $\mathbb{Z}_n$, resp., or $(i,j)$ as an element of $\mathbb{Z}_m \times \mathbb{Z}_n$.
Then simultaneous taking $\alpha^{q^i}$ and $\beta^{q^j}$ to the $q$-th power translates to adding $(1,1)$ to $(i,j)$
in $\mathbb{Z}_m \times \mathbb{Z}_n$. Thus, $(\alpha^{q^u}, \beta^{q^v})$ and $(\alpha^{q^i}, \beta^{q^j})$ belong to the same orbit
if and only if $(u,v)$ and $(i,j)$ belong to the same coset of the cyclic subgroup $U =\langle(1,1)\rangle$ in
$\mathbb{Z}_m \times \mathbb{Z}_n$.
Since the order of $U$ is $\lcm(m,n)$, the length of any orbit is also $\lcm(m,n)$.
The number of cosets of $U$ in $\mathbb{Z}_m \times \mathbb{Z}_n$ is $\gcd(m,n)$, which is then also the number of orbits.
Consequently, the values of the diamond product on any set of representatives of the orbits will determine
those for all $mn$ pairs of conjugates of $\alpha$ and $\beta$.

The generalized Chinese Remainder Theorem (see, for example, \cite[Theorem 2.4.1]{ding1996ChineseRemainderTheorem})
yields the characterization of elements of $\mathbb{Z}_m \times \mathbb{Z}_n$ belonging to the same coset of $U = \langle(1,1)\rangle$,
as stated in \cref{g_orbits}.

\begin{theorem}[Generalized Chinese Remainder Theorem]\label{g:CRT}
	Let $m,n$ be two positive integers. Then the system of congruences
	\[x \equiv a_1\pmod m,\qquad x \equiv a_2 \pmod n\]
	has solutions if and only if $\gcd(m,n)\mid a_1-a_2$. Under this condition, the above system has only one solution modulo $\lcm(m,n)$.
\end{theorem}

Combining our above discussions with the generalized Chinese Remainder Theorem we get:

\begin{lemma}\label{g_orbits}
	Let $\alpha \in \FFA_q(m)$ and $\beta \in \FFA_q(n)$.
	\begin{itemize}
		\item A diamond product on the set
		$\{ (\alpha^{q^i},\beta^{q^j})\,|\, 0\leq i\leq m-1, 0\leq j \leq n-1\}$ is uniquely determined by its values on the subset
		$\{ (\alpha^{q^i},\beta^{q^j})\,|\, (i,j) \in \mathcal{R}\}$, where $\mathcal{R}$ is a set of representatives of the
		cosets of $U = \langle(1,1)\rangle$ in $\mathbb{Z}_m \times \mathbb{Z}_n$. The sets
		$\{ (0,j) : 0\leq j \leq \gcd(m,n)-1 \}$ and $\{ (i,0) : 0\leq i \leq \gcd(m,n)-1 \}$ are such sets of representatives.
		\item For integers $u,v,i,j$, the pairs $(\alpha^{q^u}, \beta^{q^v})$ and $(\alpha^{q^i}, \beta^{q^j})$ are in the same orbit if and only if $\gcd(m,n)$ divides $(u-i)+(v-j)$.
	\end{itemize}
\end{lemma}
As already observed in Theorem 2.1 of \cite{mills2001FactorizationsRootbasedPolynomial} and its proof,  \Cref{g_orbits} and the discussions preceding it imply that irreducible factors
of $f\diamond g$ over $\FF_q$ are the minimal polynomials
of $\alpha \diamond \beta^{q^j}$ for $0\leq j \leq \gcd(m,n)$. In particular, the composed product $f\diamond g$ has at most $\gcd(m,n)$ different irreducible factors over $\FF_q$ and the degree of each irreducible factor is at most $\lcm(m,n)$.

Two natural questions arise in the study of diamond products:
\begin{enumerate}
    \item For a fixed diamond product $\diamond$, what is the set of polynomials $f,g$ such that $f\diamond g$ has a special type of factorization?
    \item For $f,g\in\FF_q[X]$ fixed, what is the set of diamond products such that $f\diamond g$ has a special type of factorization?
\end{enumerate}
Defining  diamond products on general sets $\Ga\times\Gb$ allows us to tackle both questions.
The first question is about a characterization of sets $\Ga$ and $\Gb$ for a given $\diamond$ such that the composed product of
the minimal polynomials of their elements  have a desired factorization. The second question is about the characterization
of diamond product leading to a desired factorization of $f\diamond g$ over $\FF_q$ for
a fixed pair $f,g\in\FF_q[X]$, or equivalently for the fixed sets $\Ga = \{\alpha,\alpha^q,\dots,\alpha^{q^{m-1}}\}$ and $\Gb=\{\beta,\beta^q,\dots,\beta^{q^{n-1}}\}$ of zeros of $f$ and $g$.

In this paper we study the above two questions for the diamond products satisfying
the following  conjugate cancellation rule.

\begin{definition}[Conjugate Cancellation on arbitrary sets]\label{Def:Conjugate-Cancellation}
	Let $n,m \geq 1$ be integers and $\Ga,\Gb\subseteq\FFC_q$ be Frobenius invariant sets. A diamond product $\diamond:\Ga\times\Gb\rightarrow\FFC_q$ is said to satisfy conjugate cancellation on $\Gap_m\times\Gbp_n$ if for every integer $k$ with $\gcd(m,n)\mid k$ and
	all $\alpha \in \Gap_m$, $\beta \in \Gbp_n$ it holds that
	\begin{align}
		\alpha \diamond \beta = \alpha^{q^k} \diamond \beta &\implies \alpha = \alpha^{q^k}, \label{Equ:CC-Diam-Prod-a}\\
		\alpha \diamond \beta = \alpha \diamond \beta^{q^k} &\implies \beta = \beta^{q^k} \label{Equ:CC-Diam-Prod-b}.
	\end{align}
\end{definition}
The choice  $\Ga = \FFA_q(m)$ and $\Gb = \FFA_q(n)$ with $\gcd(n,m)=1$ yields the definition of conjugate cancellation given in the introduction.
We say that the diamond product satisfies conjugate cancellation for $\alpha \in \Gap_m$ and $\beta \in \Gbp_n$, if conjugate cancellation
holds on $\Ga = \{\alpha,\alpha^q,\dots,\alpha^{q^{m-1}}\}$ and $\Gb=\{\beta,\beta^q,\dots,\beta^{q^{n-1}}\}$. By \Cref{g_orbits}~(ii), the pairs $(\alpha, \beta), (\alpha^{q^k},\beta)$ appearing in \Cref{Equ:CC-Diam-Prod-a} belong to the same orbit.
The same holds for pairs in \Cref{Equ:CC-Diam-Prod-b}. Hence, the implications in  \Cref{Equ:CC-Diam-Prod-a} and \Cref{Equ:CC-Diam-Prod-b}  need  to be checked only for pairs from the same orbit. It is worth to note that this would be not anymore the case  without the condition
$\gcd(m,n)\mid k$ in \cref{Def:Conjugate-Cancellation}.

The next example shows that  conjugate cancellation does not coincide with the concept of weak cancellation introduced in \cite{munemasa2016NoteBrawleyCarlitzTheorem}.

\begin{example}\label{Ex:CC-Less-Strict}
	Let $q = 2$ and $m = 2$ and $n = 3$. Take $f=X^2+X+1$ and $g = X^3+X+1$ and let $\alpha,\beta$ be zeros of $f,g$ respectively. Notice that the other two conjugates of $\beta$ are given by $\beta^2$ and $\beta(\beta+1)$. For $\FF_{2^6} \times \FF_{2^6}$, we define a diamond product by $x\diamond y = x\cdot u(y)$, where $u(y) = y(y+1)$. Since $u(\beta) = u(\beta+1)$, we have  $\alpha\diamond\beta = \alpha\diamond(\beta+1)$. As $\beta\neq\beta+1$, the diamond product $\diamond$ does not satisfy the weak cancellation. However, since $\beta$ and $\beta+1$ are not conjugates, this is not a contradiction to conjugate cancellation. We postpone the proof that this diamond product indeed satisfies conjugate cancellation on $\FF_{2^6} \times \FF_{2^6}$ to \Cref{Sec:Construction-Simple}.
\end{example}

The following theorem is a key step for understanding the diamond products satisfying conjugate cancellation.
\begin{theorem}\label{th:g_values}
	Let $\diamond:\Ga\times \Gb\rightarrow \FFC_q$ be a diamond product. Then it satisfies conjugate cancellation for $\alpha\in\Gap_m$ and $\beta\in\Gbp_n$ if and only if for every $0\leq j \leq  \gcd(m,n)-1$
	\[\FF_q(\alpha, \alpha \diamond \beta^{q^j}) = \FF_q(\beta, \alpha \diamond \beta^{q^j}) =\FF_q(\alpha, \beta)\]
    holds.
	Equivalently, the diamond product satisfies conjugate cancellation for $\alpha\in\Gap_m$ and $\beta\in\Gbp_n$ if and only if for every $0\leq j \leq \gcd(m,n)-1$ we have
	$\alpha\diamond\beta^{q^j} \in\FFA_q(r_j)$ with $r_j$ satisfying
	\[\lcm(m,n) = \lcm(n,r_j) = \lcm(m,r_j).\]
\end{theorem}
\begin{proof}
	Using symmetry, we prove the statement only for $j=0$. Suppose $\alpha \diamond \beta \in \FF_{q^r}$. Then
	$r$ is a divisor of $\lcm(m,n)$ since
	\begin{equation*}
		(\alpha \diamond \beta)^{q^{\lcm(m,n)}}=\alpha^{q^{\lcm(m,n)}}\diamond \beta^{q^{\lcm(m,n)}}=\alpha \diamond \beta.
	\end{equation*}
	Suppose we have
	\begin{equation}\label{g:eq_k}
		\alpha \diamond \beta = \alpha^{q^{k}} \diamond \beta
	\end{equation}
	for some $k$ which is divisible by $\gcd(m,n)$. By \Cref{g_orbits}~(ii),
	$(\alpha, \beta)$ and $(\alpha^{q^{k}}, \beta)$ are in the same orbit of the diamond product. Hence, there is an integer $t$ satisfying
	\[\alpha \diamond \beta = \alpha^{q^{k}} \diamond \beta = (\alpha \diamond \beta)^{q^t} = \alpha^{q^t} \diamond \beta^{q^t}.\]
	The equality $\alpha \diamond \beta = (\alpha \diamond \beta)^{q^t}$ holds if and only if $t \equiv 0 \pmod r$.
	Further, $\alpha^{q^t} = \alpha^{q^k}$ if and only if $t\equiv k \pmod m$. And finally,
	$\beta^{q^t} = \beta$ if and only if   $t \equiv 0 \pmod n$.
	These three congruences are equivalent to
	\[t\equiv k\pmod m, ~~~~~t \equiv 0\pmod{\lcm(n,r)}.\]
	By \cref{g:CRT}, the latter congruences have a solution if and only if
	$\gcd(m, \lcm(n,r))$ divides $k$. This shows that we find a $k$   fulfilling \cref{g:eq_k} and $\alpha^{q^k} \ne \alpha$ (and hence
	violating the implication in \cref{Equ:CC-Diam-Prod-a})
	if and only if $\gcd(m, \lcm(n,r)) \ne m$.
	Similar arguments show that the equation $\alpha \diamond \beta^{q^k} = \alpha \diamond \beta$ holds
	for some $k$ with $\beta^{q^k} \ne \beta$ if and only if $\gcd(n, \lcm(m,r)) \ne n$.
	Combining these two observations, we get that conjugate cancellation  holds on the orbit of
	$(\alpha, \beta)$ if and only if $m$ divides $\lcm(n,r)$ and $n$ divides $\lcm(m,r)$.
	Recall that $r$ must divide $\lcm(m,n)$, too. Now it is easy to see that these divisibility
	properties imply
	\[\lcm(m,n) = \lcm(n,r) = \lcm(m,r),\]
	completing the proof.
\end{proof}

\Cref{th:g_values} generalizes Theorem 3.1  from \cite{mills2001FactorizationsRootbasedPolynomial}, which describes possible degrees of irreducible factors of
composed products induced by a diamond product defining a
group structure on a set $G$. Recall that such a diamond product satisfies conjugate cancellation, but the converse
is not true. It is worth to note, that the set $H_{m,n}$,
appearing in \cite{mills2001FactorizationsRootbasedPolynomial} having a rather technical definition coincides
with the set of permitted values $r$ in \Cref{th:g_values}.

\subsection{Irreducible diamond products of polynomials}
In order for the composed product of the minimal polynomials of $\alpha \in \FFA_q(m)$ and $\beta \in\FFA_q(n)$
to be irreducible, it is necessary that $\gcd(m,n)=1$, since $\alpha \diamond \beta \in \FF_{q^{\lcm(n,m)}}$. The results from
\cite{brawley1987IrreduciblesComposedProduct,munemasa2016NoteBrawleyCarlitzTheorem, irimagzi2023DiamondProductsEnsuring, krompholz} show that conjugate cancellation is a sufficient condition
for constructing irreducible
polynomials using diamond products. \cref{th:g_values} implies that
it is a necessary condition as well:

\begin{corollary}\label{g:cor-irred}
	Let $\gcd(m,n)=1$. Then a diamond product satisfies conjugate cancellation for $\alpha \in \FFA_q(m)$ and $\beta\in\FFA_q(n)$ if and only
	$\alpha \diamond \beta \in \FFA_q(mn)$.
\end{corollary}
\begin{proof}
	The statement is a direct consequence of \cref{th:g_values}, since in this case only one orbit exists and $r=mn$ satisfies
	$\lcm(r,m) = \lcm(r,n)=\lcm(m,n)$.
\end{proof}

In the following theorem we summarize the obtained results in the case $\gcd(n,m)=1$, leading
to constructions of irreducible polynomials.

\begin{theorem}[Generalized Brawley--Carlitz--Theorem]\label{Thm:General-Brawley-Carlitz}
	Let $m,n\in\NN$ and let $\diamond:\Ga\times\Gb\rightarrow\FFC_q$ be a diamond product. Then the following statements are equivalent:
	\begin{enumerate}
		\item The element $\alpha\diamond\beta$ belongs to $\FFA_q(mn)$ for any $\alpha\in\Gap_m$ and $\beta\in\Gbp_n$.
		\item The polynomial $f\diamond g$ is irreducible over $\FF_q$ for any
		irreducible $f,g\in\FF_q[X]$ with their  zeros belonging to $\Gap_m$ and $\Gbp_n$, resp.
		\item $\gcd(m,n) = 1$ and $\diamond$ satisfies conjugate cancellation on $\Gap_m\times\Gbp_n$.
	\end{enumerate}
\end{theorem}

Next we present some observations on the factorization of irreducible composed products over  intermediate fields.
These results are motivated by the following lemma.

\begin{lemma}\label{Lma:irred-factor-over-ext-field}\cite{kyuregyan2011IrreducibleCompositionsPolynomials}
	A polynomial $f\in\FF_q[X]$ of degree $m = dk$ is irreducible over $\FF_q$ if and only if there is a monic irreducible polynomial $f_0(X) = X^d + \sum_{j=0}^{d-1} v_jX^j$ over $\FF_{q^k}$ such that $\FF_q(v_0,\dots,v_{d-1}) = \FF_{q^k}$ and
	\begin{equation}\label{eq:g_fact_inf}
	f(X) = \prod_{\mu=0}^{k-1}f_0^{(\mu)}(X) \in \FF_{q^k}[X],
	\end{equation}
	where $f_0^{(\mu)}$ is defined as
	\[f_0^{(\mu)}=X^d + \sum_{j=0}^{d-1} v_j^{q^\mu}X^j. \]
\end{lemma}

Observe, that the factorization \cref{eq:g_fact_inf} is clearly unique and if $\alpha \in  \FF_{q^m}$ is a zero of $f_0$, then
\[f_0 = \prod_{j=0}^{d-1}(X-\alpha^{q^{kj}}) \in \FF_{q^m}[X],\]
that is, the zeros of $f_0$ are the conjugates of $\alpha$ over $\FF_{q^k}$. Consequently, the zeroes
of $f_0^{(\mu)}$ are the conjugates of $\alpha^{q^{\mu}}$ over $\FF_{q^k}$ for every $0\leq \mu \leq k-1$.
Further, note that any irreducible factor of $f$ in $\FF_{q^k}[X]$ has degree $d$. The opposite is also true,
any polynomial of degree $d$ dividing $f$ in $\FF_{q^k}[X]$ is irreducible.

It was already observed in \cite{brawley1987IrreduciblesComposedProduct} that if $f\in\FF_q[X]$ is monic and reducible, i.e., $f = f_1\cdot f_2$ with $f_1,f_2\in\FF_q$, the composed product with the monic polynomial $g\in\FF_q[X]$ satisfies $f\diamond g = (f_1\diamond g)\cdot(f_2\diamond g)$. Using the lemma above, we can extend this idea to factorization over intermediate fields.

\begin{theorem}
Let $\gcd(m,n) = 1$, $k\geq 1$ be a divisor of $m$ and $\ell \geq 1$ be a divisor of $n$. Further,
	let $f,g\in\FF_q[X]$ be monic irreducible polynomials with $\deg(f) = m$, $\deg(g) = n$ and
	 $f = f_0^{(0)}\cdots f_0^{(k-1)} \in \FF_{q^k}[X]$ and $g = g_0^{(0)}\cdots g_0^{(\ell-1)} \in \FF_{q^\ell}[X]$.
	Then a diamond product $\diamond$ satisfies conjugate cancellation for the zeros of
	$f$ and $g$ if and only if
	\begin{equation}\label{Equ:factorization}
		f\diamond g = \prod_{\mu=0}^{k-1}\prod_{\nu =0}^{\ell-1} f_0^{(\mu)}\diamond g_0^{(\nu)} \in \FF_{q^{k\ell}}[X],
	\end{equation}
	where the polynomials $f_0^{(\mu)}\diamond g_0^{(\nu)}$ are irreducible in $\FF_{q^{k\ell}}[X]$ and the coefficients of
	$f_0^{(\mu)}\diamond g_0^{(\nu)}$ generate the field $\FF_{q^{k\ell}}$.
\end{theorem}
\begin{proof} Let $\alpha \in \FF_{q^m}$ with $f_0 = \prod_{i=0}^{(m/k)-1}(X-\alpha^{q^{ki}})$, and
let $\beta \in \FF_{q^n}$ with $g_0 = \prod_{j=0}^{(n/\ell)-1}(X-\beta^{q^{\ell j}})$.
Then
\[
f_0 \diamond g_0 = \prod_{i=0}^{(m/k)-1}\prod_{j =0}^{(n/\ell)-1}(X-\alpha^{q^{ki}} \diamond \beta^{q^{\ell j}}).
\]
Since $\gcd(k,\ell)=1$, the polynomials $f_0$ and $g_0$ remain irreducible over the extension field $\FF_{q^{k\ell}}$.

If $\diamond$ satisfies conjugate cancellation, then by \Cref{Thm:General-Brawley-Carlitz}, we have that $f_0 \diamond g_0$ is irreducible over $\FF_{q^{k\ell}}$ and  $f\diamond g$ is irreducible over $\FF_q$. Then \cref{Lma:irred-factor-over-ext-field} implies
\[f\diamond g = \prod_{\delta = 0}^{kl-1}(f_0\diamond g_0)^{(\delta)} =
\prod_{\mu=0}^{k-1}\prod_{\nu =0}^{\ell-1} f_0^{(\mu)}\diamond g_0^{(\nu)},\]
since
\[(f_0\diamond g_0)^{(\delta)} = f_0^{(\delta)}\diamond g_0^{(\delta)} =
f_0^{(\mu)}\diamond g_0^{(\nu)},\]
with $\mu \equiv \delta {\pmod k}$ and $\nu \equiv \delta {\pmod \ell}$.

To prove the opposite, we use again the Chinese Remainder Theorem
\[f\diamond g = \prod_{\mu=0}^{k-1}\prod_{\nu =0}^{\ell-1} f_0^{(\mu)}\diamond g_0^{(\nu)} = \prod_{\delta = 0}^{kl-1}(f_0\diamond g_0)^{(\delta)},\]
and then \cref{Lma:irred-factor-over-ext-field} completes the proof.
\end{proof}

Important special cases of the above result are the case $(k,l)=(1,n)$ or $(k,l)=(m,1)$, which yield the factorization of $f\diamond g$ over
$\FF_{q^n}$ and $\FF_{q^m}$ respectively. We state here the corresponding statement for $(k,l)=(m,1)$:

\begin{corollary}
	Let $f,g\in\FF_q[X]$ be irreducible with $\deg(f) = m$ and $\deg(g) = n$, where $\gcd(m,n) = 1$.
	Let a diamond product $\diamond$ be defined on the zeroes of $f$ and $g$. Set
	\[h = \prod_{j=0}^{n-1}(X-\alpha\diamond\beta^{q^j}).\]
	Then $h\in\FF_{q^m}[X]$ and $f\diamond g = \prod_{\mu=0}^{m-1} h^{(\mu)}$. Moreover,
	$\diamond$ satisfies conjugate cancellation if and only if $h$ is irreducible and
	its coefficients generate the field $\FF_{q^m}$.
\end{corollary}

\subsection{Decomposition of \texorpdfstring{$f \diamond g$}{f⬦g} over \texorpdfstring{$\FF_q$}{Fq}}\label{sub_sec:decopmosition}

For a given diamond product and two elements $\alpha \in \FFA_q(m)$ and $\beta \in \FFA_q(n)$,
\cref{g_orbits} implies that over $\FF_q$ the diamond product of their minimal polynomials is a product
of minimal polynomials of $\alpha \diamond \beta^{q^j}$, where $0\leq j \leq \gcd(m,n)-1$. We denote by
$m_\gamma$ the minimal polynomial of $\gamma$ over $\FF_q$. Then we have
\begin{equation}\label{g:eq_prod}
	m_\alpha\diamond m_\beta = \prod_{j=0}^{\gcd(m,n)-1}m_{\alpha\diamond \beta^{q^j}}^{\lcm(m,n)/r_j},
\end{equation}
where the number $r_j$ is the degree of the extension $\FF_q(\alpha\diamond \beta^{q^j})$ over $\FF_q$.
Moreover, by \cref{th:g_values} the diamond product $\diamond$ fulfills conjugate cancellation for $\alpha$ and $\beta$
if and only if every $r_j$ appearing in the factorization \eqref{g:eq_prod} satisfies $\lcm(r_j,m)=\lcm(r_j,n)=\lcm(m,n)$.
Note that in general it is possible that the
$m_{\alpha\diamond \beta^{q^j}}$ coincide for different $j$ with $0\leq j  <\gcd(m,n)$.

In the remainder of this section we characterize numbers $r$ satisfying $\lcm(r,m)=\lcm(r,n)=\lcm(m,n)$ in terms of prime decomposition
of $m$ and $n$. Further, we use it to describe the pairs $m,n$ for which the decomposition \cref{g:eq_prod} contains only factors
of maximal possible degree $\lcm(m,n)$ for every $(\alpha, \beta) \in \FFA_q(m) \times \FFA_q(n)$.  \\

For a prime $p$ and an integer $a$, we define the integer $\nu_p(a) \geq 0$ such that
$p^{\nu_p(a)} \mid a$ and $p^{\nu_p(a)+1} \nmid a$.

\begin{lemma}\label{g:lem_char_r}
	For given integers $n$ and $m$ define the pairwise coprime divisors $o,n_1,n_2$ of $n$, and resp. the pairwise coprime divisors
	$o, m_1,m_2$ of $m$, to satisfy:
	\[n=o\cdot n_1\cdot n_2, ~~ m=o\cdot m_1\cdot m_2\]
	such that $\nu_p(n) >\nu_p(m)$ for any prime divisor $p$ of $n_1$ and $\nu_{p'}(m) > \nu_{p'}(n)$ for any prime divisor $p'$
	of $m_1$. Then a divisor $r$ of $\lcm(n,m) = o\cdot n_1 \cdot m_1$ fulfills $\lcm(r,m)=\lcm(r,n)=\lcm(m,n)$ if and only if
	$r=o'\cdot n_1 \cdot m_1$ with $o'$ dividing $o$. In particular, there is no nontrivial divisor $r$ of $\lcm(m,n)$ satisfying
	the condition if and only if $o=1$, or equivalently if $\nu_s(m) \ne \nu_s(n)$ for any prime $s$ dividing $m\cdot n$.
\end{lemma}
\begin{proof} The statement follows directly from the prime factor decomposition of the numbers $m,n,r$.
\end{proof}

\begin{theorem}\label{Thm:Brawley-Carlitz-gcd}
	Let $m,n\in\NN$ be such that $\nu_p(m) \neq \nu_p(n)$ for any prime $p$ dividing $m \cdot n$, and let $\diamond:\Ga\times\Gb\rightarrow\FFC_q$ be a diamond product. Then the following statements are equivalent:
	\begin{enumerate}
		\item For every $(\alpha,\beta)\in\Gap_m \times\Gbp_n$ it holds that $\alpha\diamond\beta\in\FFA_q(\lcm(m,n))$.
		\item For every $(\alpha,\beta)\in\Gap_m \times\Gbp_n$, any irreducible factor of $m_\alpha\diamond m_\beta$ has degree $\lcm(m,n)$.
		\item $\diamond$ satisfies conjugate cancellation on $\Gap_m\times\Gbp_n$.
	\end{enumerate}
\end{theorem}
\begin{proof}
	The statement follows from \Cref{th:g_values}, \Cref{g:eq_prod} and \Cref{g:lem_char_r}.
\end{proof}

The next corollary specifies the factorization of \Cref{Thm:Brawley-Carlitz-gcd}~(ii):

\begin{corollary} \label{cor:213}
	Let $m,n\in\NN$ such that $\nu_p(m) \neq \nu_p(n)$ for any prime $p$ dividing $m \cdot n$, and let $\diamond:\Ga\times\Gb\rightarrow\FFC_q$ be a diamond product satisfying conjugate cancellation on $\Gap_m\times\Gbp_n$. Then for every $(\alpha,\beta)\in\Gap_m \times\Gbp_n$ it holds
	\begin{align*}
	m_\alpha\diamond m_\beta
		&=\prod_{i=0}^{\gcd(m,n)-1}{\prod_{j=0}^{\lcm(m,n)-1}(X-\alpha^{q^j}\diamond\beta^{q^{i+j}})}\\
		&=\prod_{i=0}^{\gcd(m,n)-1}{\prod_{j=0}^{\lcm(m,n)-1}(X-\alpha^{q^{i+j}}\diamond\beta^{q^{j}})},
	\end{align*}
	and
	\begin{equation*}
		m_{\alpha\diamond\beta^{q^i}} = \prod_{j=0}^{\lcm(m,n)-1}(X-\alpha^{q^j}\diamond\beta^{q^{i+j}})\quad\text{and}\quad
		m_{\alpha^{q^i}\diamond\beta}=\prod_{j=0}^{\lcm(m,n)-1}(X-\alpha^{q^{i+j}}\diamond\beta^{q^{j}}).
	\end{equation*}
	In particular, $m_{\alpha\diamond\beta^{q^i}}$ has  $\gcd(m,n)$ distinct irreducible factors over $\FF_q$ if and only if
	for every $i \ne i',~0\leq i,i' \leq \gcd(m,n)-1$ the elements $\alpha\diamond\beta^{q^i}$ and $\alpha\diamond\beta^{q^{i'}}$
	are not conjugates over $\FF_q$.
\end{corollary}

\section{Characterization and construction of diamond products satisfying conjugate cancellation}\label{Sec:Construction}

Theorem~\ref{Thm:Brawley-Carlitz-gcd} and Corollary~\ref{cor:213} show that the factorization of composed products can be explicitly determined if the diamond product satisfies conjugate cancellation. In this section, we give several criteria that ensure that a diamond product satisfies conjugate cancellation. Via Theorem~\ref{Thm:Brawley-Carlitz-gcd}, this can then for example be used to explicitly construct many irreducible polynomials. We give concrete examples later in this section.

We study a natural choice of diamond products, namely diamond products that are described by bivariate polynomials over $\FF_q$, so-called $\phi$-products. More precisely, we call a diamond product $\diamond:\Ga\times\Gb\rightarrow\FFC_q$ a $\phi$-product, if there is a $\phi\in\FF_q[X,Y]$ such that
\[ \alpha\diamond\beta = \phi(\alpha,\beta) \]
for all $(\alpha,\beta)\in\Ga\times\Gb$. Conversely, every polynomial $\phi\in\FF_q[X,Y]$ induces a diamond product on $\FFC_q\times\FFC_q$. These products were already studied in \cite{brawley1987IrreduciblesComposedProduct,stichtenoth2014NoteComposedProducts}. Algorithms for computing composed products based on $\phi$-products can be found in \cite{brawley1999computing,bostan2006FastComputationSpecial}. It is a natural question which choices of $\phi$ yield diamond products satisfying conjugate cancellation.

In this section, it is advantageous to represent the polynomials $\phi\in\FF_q[X,Y]$ in the following general form:
\begin{equation}\label{Equ:Diam-Prod-Repr}
	\phi(X,Y) = \sum_{s=1}^r u_s(X)v_s(Y),
\end{equation}
where $u_s\in\FF_q[X]$ and $v_s\in\FF_q[Y]$ are polynomials.
This representation generalizes many forms of polynomial representation.
For example, let $m,n$ be arbitrary integers and consider the following three representations:
\[
	\phi(X,Y) = \sum_{i=0}^{m-1}\sum_{j=0}^{n-1}c_{ij}X^iY^j = \sum_{i=0}^{m-1} \chi_i(Y)X^i = \sum_{j=0}^{n-1} \psi_j(X)Y^j
\]
for some coefficients $c_{ij}\in\FF_q$, and coefficient polynomials $\chi_i\in\FF_q[Y]$ and $\psi_j\in\FF_q[X]$ defined in the obvious way. All three of them are instances of \Cref{Equ:Diam-Prod-Repr}. For example, we obtain the last one by letting $r = n$, $v_s(Y) = Y^{s-1}$, and $u_s(X) = \sum_{i=0}^{m-1} c_{i,s-1}X^i$ for $s = 1,\dots,n$.

The main advantage that comes with \Cref{Equ:Diam-Prod-Repr} is that the coefficient polynomials can be chosen flexibly. Ultimately, this leads to the applicability of our methods to other classes of polynomials, e.g., $\phi(X,Y) = \sum_{i=0}^{m-1}\sum_{j=0}^{n-1}c_{ij}X^{q^i}Y^{q^j}$, which are the subject of \Cref{Sec:Normal}.

\subsection{Two key lemmas}

We now consider the following question: When does the $\phi$-product of a given $\phi\in\FF_q[X,Y]$ satisfy conjugate cancellation on $\Gap_m\times\Gbp_n$? In other words, we need to check if for every $k$ with $\gcd(m,n)\mid k$ it holds that
\begin{align}
	\phi(\alpha,\beta) = \phi(\alpha^{q^k},\beta) &\implies \alpha = \alpha^{q^k}, \label{Equ:CC-Phi-Prod-a}\\
	\phi(\alpha,\beta) = \phi(\alpha,\beta^{q^k}) &\implies \beta = \beta^{q^k} \label{Equ:CC-Phi-Prod-b}
\end{align}
where $(\alpha,\beta)\in \Gap_m\times\Gbp_n$. In some of the later proofs we consider only  \Cref{Equ:CC-Phi-Prod-a}. We refer to this as the right-sided conjugate cancellation, as we cancel out the argument on the right-hand side. It is clear from the symmetry in the representation \Cref{Equ:Diam-Prod-Repr} that analogous statements can be proven for \Cref{Equ:CC-Phi-Prod-b} as well by switching the roles of $u_s$ and $v_s$.

The following \namecref{Lma:Coeff-Polys} generalizes \cite[Lemma 8]{munemasa2016NoteBrawleyCarlitzTheorem}.

\begin{lemma}\label{Lma:Coeff-Polys}
	Let  $m,n>1$, $~k\geq 1$, $\beta \in \FFA_q(n)$ and $\phi\in\FF_q[X,Y]$ have a representation
	\[ \phi(X,Y) = \sum_{s=1}^r u_s(X)v_s(Y). \]
	Further, let $\{v_s(\beta) \mid s=1,\dots,r\}$ be linearly independent over $\FF_{q^m}$ and $\alpha\in\FFA_q(m)$. Then  $\phi(\alpha^{q^k},\beta) = \phi(\alpha,\beta)$ if and only if $u_s(\alpha) \in \FF_{q^k}$ for every $s = 1,\dots,r$.
\end{lemma}
\begin{proof}
	Assuming that $\phi(\alpha^{q^k},\beta) = \phi(\alpha,\beta)$, simple calculation shows
	\begin{align*}
		0 &= \phi(\alpha^{q^k},\beta) - \phi(\alpha,\beta)\\
		&= \sum_{s=1}^r \left(u_s(\alpha)-u_s(\alpha^{q^k})\right) v_s(\beta)\\
		&= \sum_{s=1}^r \left(u_s(\alpha)-u_s(\alpha)^{q^k}\right) v_s(\beta),
	\end{align*}
 	where we use that $u_s\in\FF_q[X]$ and hence $u_s(\alpha) \in \FF_{q^m}$.
	Because of the linear independence of the elements $v_s(\beta)$ over $\FF_{q^m}$, we must have $u_s(\alpha) = u_s(\alpha)^{q^k}$ for all $s = 1,\dots,r$. The converse is obvious.
\end{proof}

\begin{lemma}\label{Lma:Main-Lemma}
	Let  $m,n>1$ and $\phi\in\FF_q[X,Y]$ have a representation \[ \phi(X,Y) = \sum_{s=1}^r u_s(X)v_s(Y). \]
    Further, let $\beta\in\FFA_q(n)$ such that $\{v_s(\beta) \mid s=1,\dots,r\}$ are linearly independent over $\FF_{q^m}$ and $\alpha\in\FFA_q(m)$. Then  the implication
	\begin{equation}\label{Equ:Lma32-Implication}
		\phi(\alpha,\beta) = \phi(\alpha^{q^k},\beta)\implies \alpha = \alpha^{q^k}
	\end{equation}
	holds for every $k\in\ZZ$ if and only if $\FF_q(u_1(\alpha),u_2(\alpha),\dots,u_r(\alpha)) = \FF_{q^m}$.
\end{lemma}
\begin{proof}
	Assume that Eq. \eqref{Equ:Lma32-Implication} holds. Let $k\in\NN$ such that $\FF_q(u_1(\alpha),\dots,u_r(\alpha)) = \FF_{q^k}$. We claim that $k = m$. Note that, by construction, $\FF_{q^k}$ is a subfield of $\FF_{q^m}$. Therefore, we have $k \leqslant m$. For $\beta\in\FFA_q(n)$ the following equation holds:
	\begin{align*}
		\phi(\alpha,\beta)-\phi(\alpha^{q^k},\beta)
		&= \sum_{s=1}^r\left(u_s(\alpha)-u_s(\alpha^{q^k})\right)v_s(\beta) \\
		&= \sum_{s=1}^r\left(u_s(\alpha)-u_s(\alpha)^{q^k}\right)v_s(\beta)\\
		&= 0,
	\end{align*}
    where we use that $u_s\in\FF_q[X]$ and $u_s(\alpha)\in\FF_{q^k}$.
	From Eq. \eqref{Equ:Lma32-Implication} it follows that $\alpha = \alpha^{q^k}$. This implies $m\mid k$ because $\alpha\in\FFA_q(m)$. Together with $k\leqslant m$ we arrive at $k = m$.

	To prove the converse, take $k\in\ZZ$ such that $\phi(\alpha,\beta) = \phi(\alpha^{q^k},\beta)$. Without loss of generality we may assume $k>0$. From \Cref{Lma:Coeff-Polys} it follows that
	\[u_s(\alpha) = u_s(\alpha)^{q^k} \text{ for all } s = 1,\dots,r.\]
	Therefore, $u_s(\alpha)\in\FF_{q^k}$ for all indices $s$. Now the assumption implies
	\[ \FF_{q^m} = \FF_q(u_1(\alpha),u_2(\alpha),\dots,u_r(\alpha)) \subseteq \FF_{q^k}. \]
	In particular, $\FF_{q^m}$ is a subfield of $\FF_{q^k}$, so $m\mid k$. Since $\alpha\in\FFA_q(m)$, we arrive at $\alpha = \alpha^{q^k}$.
\end{proof}

\begin{remark}
	\begin{enumerate}
		\item Notice that \Cref{Lma:Coeff-Polys,Lma:Main-Lemma} do not put any restriction on the degrees of the coefficient polynomials $u_s(X)$ and $v_s(Y)$ nor on the value of $\gcd(m,n)$. These lemmas both play an important role in the following sections where we impose different additional conditions.
		\item The implication \eqref{Equ:Lma32-Implication} shows that the value of $\phi(\alpha,\beta)$ is different from $\phi(\alpha^{q^k},\beta)$ for $k = 1,\dots,m-1$. Indeed, applying \Cref{Lma:Main-Lemma} to $\alpha^q,\dots,\alpha^{q^{m-1}}$ and noticing that $u_s(\alpha)^{q^k}$ belongs to the same field extension as $u_s(\alpha)$ shows that the values
		\[ \phi(\alpha,\beta),~\phi(\alpha^q,\beta),\dots,\phi(\alpha^{q^{m-1}},\beta) \]
		are pairwise distinct.
	\end{enumerate}
\end{remark}

\subsection{An important class of \texorpdfstring{$\phi$}{phi}-products}
With the help of \Cref{Lma:Main-Lemma}, we now show a characterization of conjugate cancellation on $\Gap_n\times\Gbp_m$ for a special kind of $\phi$-products. We assume that the degrees of the extension fields satisfy $\gcd(m,n) = 1$, and the partial degrees of $\phi$ satisfy $\deg_X(\phi) < m$ and $\deg_Y(\phi) < n$.

This type of $\phi$-product is of particular importance, as for every choice of monic irreducible polynomials $f,g,h\in\FF_q[X]$ with $\deg(f) = m$, $\deg(g) = n$, and $\deg(h) = mn$, there exists a $\phi$-product satisfying $\deg_X(\phi) < m$, $\deg_Y(\phi) < n$ and conjugate cancellation on the roots of $f,g$ such that $f\diamond_\phi g = h$. Indeed, let $\alpha$ be a root of $f$ and $\beta$ be a root of $g$. It is easy to see that the elements $\alpha^i\beta^j$ for $0\leq i\leq m-1$ and $0\leq j\leq n-1$ form a basis of $\FF_{q^{mn}}$ over $\FF_q$. Letting $\gamma$ be a root of $h$, there is a unique representation $\gamma = \sum_{ij}c_{ij}\alpha^i\beta^j$ with $c_{ij}\in\FF_q$. Define $\phi(X,Y) = \sum_{ij}c_{ij}X^iY^j$ for which $\phi(\alpha,\beta) = \gamma$ holds. Note that $\phi$ satisfies conjugate cancellation on $\alpha$ and $\beta$ by \Cref{Thm:General-Brawley-Carlitz}.

The next theorem describes a choice for coefficient polynomials $u_s$ and $v_s$ for which the linear independence of the evaluations on $\alpha$ resp. $\beta$ is automatically satisfied and hence \Cref{Lma:Main-Lemma} applies.

\begin{theorem}\label{Thm:Conj-Cancel}
	Let integers $m,n>1$ be coprime. Let $\phi\in\FF_q[X,Y]$ have a representation
	\[ \phi(X,Y) = \sum_{s=1}^r u_s(X)v_s(Y), \qquad \deg(u_s) < m, ~ \deg(v_s) < n, \]
	where the sets of coefficient polynomials $\{u_s\in\FF_q[X]\mid s=1,\dots,r\}$ and $\{v_s\in\FF_q[Y]\mid s= 1,\dots,r\}$ are linearly independent over $\FF_q$.
	The $\phi$-product satisfies conjugate cancellation on $(\alpha,\beta)\in\Gap_m\times\Gbp_n$ if and only if the following properties hold
	\begin{enumerate}
		\item $\FF_q(u_1(\alpha),u_2(\alpha),\dots,u_r(\alpha)) = \FF_{q^{m}}$, and
		\item $\FF_q(v_1(\beta),v_2(\beta),\dots,v_r(\beta)) = \FF_{q^{n}}$.
	\end{enumerate}
\end{theorem}
\begin{proof}
	Let $\beta\in\FFA_q(n)$ and $g\in\FF_q[X]$ be the minimal polynomial of $\beta$. From \cite[Theorem 3.46]{lidl1996FiniteFields} it follows that $g$ is also irreducible over $\FF_{q^m}$. Therefore, $\beta\in\FFA_{q^m}(n)$. Hence, $\FFA_q(n)\subseteq\FFA_{q^m}(n)$.

	We show that the set of evaluations $\{ v_s(\beta) : s = 1,\dots,r\}$ is linearly independent over $\FF_{q^m}$. Let $c_s\in\FF_{q^m}$ such that
	\[ 0 = \sum_{s=1}^r c_sv_s(\beta) = \left(\sum_{s=1}^r c_sv_s \right)(\beta). \]
	Since $\beta\in\FFA_{q^m}(n)$, the minimal polynomial of $\beta$ must be of degree $n$. Additionally, since $\deg(v_s) \leqslant n-1$, it follows that $\sum_{s=1}^r c_sv_s = 0$. Consider the $r\times n$ matrix $A = (v_{sj})_{sj}$, where $v_{sj}\in\FF_q$ is the coefficient of $Y^j$ in $v_s$. Now $\sum_{s=1}^r c_sv_s = 0$ is equivalent to
	\begin{equation}\label{Equ:Thm34}
		(0,\dots,0) = (c_1,\dots, c_r)A.
	\end{equation}
	Since the polynomials $v_1,\dots,v_r$ are linearly independent, it follows that the rows of $A$ are linearly independent over $\FF_q$. Equivalently, $A$ has an $r\times r$ submatrix with non-vanishing determinant. The determinant is independent of the field extension we view $A$ in. Hence, the rows of $A$ are linearly independent over $\FF_{q^m}$, and \Cref{Equ:Thm34} implies $c_1,\dots,c_r = 0$. In other words, $v_1(\beta),\dots,v_r(\beta)$ are linearly independent over $\FF_{q^m}$.

	Applying \Cref{Lma:Main-Lemma} shows that $\phi$ satisfies \eqref{Equ:CC-Phi-Prod-a} on $(\alpha,\beta)$ if and only if it holds that
	$\FF_q(u_1(\alpha),u_2(\alpha),\dots,u_r(\alpha)) = \FF_{q^{m}}$.
	The equivalence of \eqref{Equ:CC-Phi-Prod-b} and (ii) can be shown in the same way by swapping the roles of $u_s$ and $v_s$.
\end{proof}

\begin{remark}\label{Rmk:Existence-Special-Repr}
	Given $\phi\in\FF_q[X,Y]$, it is always possible to choose $\{u_1,\dots,u_r\}$ and $\{v_1,\dots,v_r\}$ in such a way that both sets of polynomials  are linearly independent over $\FF_q$.

	Let $\phi(X,Y) = \sum_{i=0}^{m-1}\sum_{j=0}^{n-1}c_{ij}X^iY^j$. Then define $C = (c_{ij})_{ij}\in\FF_q^{m\times n}$ and let $r = \rank(C)$. Now, we can decompose $C$ into a sum of rank 1 matrices, i.e., $C = \sum_{s = 1}^r a_sb_s^T$,
	where $a_1,\dots,a_r\in\FF_q^m$ and $b_1,\dots,b_r\in\FF_q^n$ are two linearly independent collections of vectors.
	Let $u_s(X) = \sum_{i=0}^{m-1} a_{si} X^i$ and $v_s(Y) = \sum_{j=0}^{n-1} b_{sj}Y^j$, where $a_{si}$ denotes the $i$-th entry of $a_s$ and $b_{sj}$ the $j$-th entry of $b_s$. Then $c_{ij} = \sum_{s=1}^r a_{si}b_{sj}$, which implies
	\[\phi(X,Y) = \sum_{i=0}^{m-1}\sum_{j=0}^{n-1}\sum_{s=1}^r a_{si}b_{sj} X^iY^j = \sum_{s=1}^r u_s(X)v_s(Y).\]
	The linear independence of $u_1,\dots,u_r$ and $v_1,\dots,v_r$ is obvious. Conversely, if $\phi$ is represented as in \Cref{Equ:Diam-Prod-Repr} with linear independent $u_1,\dots,u_r$ and $v_1,\dots,v_r$, the same calculation shows that $C$ is the sum of $r$ rank 1 matrices. This implies $\rank(C) = r$.
\end{remark}

We now give a characterization for conjugate cancellation using a very common representation of bivariate polynomials. Note that the coefficient polynomials $\chi_i$ and $\psi_j$ do not need to be linearly independent.

\begin{theorem}\label{Cor:Conj-Cancel-nice-coeffs}
	Let integers $m,n>1$ be coprime and let $\phi\in\FF_q[X,Y]$ have coefficient polynomials $\chi_i\in\FF_q[X],~\psi_j\in\FF_q[Y]$ such that
	\[ \phi(X,Y) = \sum_{i=0}^{m-1} \chi_i(Y)X^i = \sum_{j=0}^{n-1} \psi_j(X)Y^j. \]
	The $\phi$-product satisfies conjugate cancellation on $(\alpha,\beta)\in\Gap_m\times\Gbp_n$ if and only if of the following properties hold
	\begin{enumerate}
		\item $\FF_q(\psi_0(\alpha),\psi_1(\alpha),\dots,\psi_{n-1}(\alpha)) = \FF_{q^{m}}$, and
		\item $\FF_q(\chi_0(\beta),\chi_1(\beta),\dots,\chi_{m-1}(\beta)) = \FF_{q^{n}}$.
	\end{enumerate}
\end{theorem}
\begin{proof}
	Note that the representation of $\phi$ already implies that $\deg(\chi_i) < n$ and $\deg(\psi_j) < m$. Let $v_s(Y) = Y^s$ and $u_s(X) = \psi_s(X)$ for $s = 1,\dots,n-1 =: r$. Since $1,\beta,\beta^2,\dots,\beta^{n-1}$ are linearly independent over $\FF_{q^m}$, it follows form \Cref{Lma:Main-Lemma} that $\phi$ satisfies right-sided conjugate cancellation from \eqref{Equ:CC-Phi-Prod-a} on $(\alpha,\beta)$ if and only if $\FF_q(\psi_0(\alpha),\psi_1(\alpha),\dots,\psi_{n-1}(\alpha)) = \FF_{q^{m}}$.
	Again, the second part follows analogously.
\end{proof}

\subsection{Algorithmic verification of conjugate cancellation}\label{Sec:Algorithm}

We now give an algorithm that verifies conjugate cancellation. We use the arithmetic in $R = \FF_q[X]/(f)$ for a polynomial $f\in\FF_q[X]$. Let $\mathsf{M}(m)$ be the number of operations in $\FF_q$ required to multiply two polynomials in $\FF_q[X]$ with degree less than $m$. For instance, classical multiplication requires $\mathsf{M}(m) = \bigO(m^2)$ operations. Faster algorithms only require $\mathsf{M}(m) = \bigO(m\log(m)\log(\log(m)))$ operations (see e.g., \cite[Ch. 8.3]{vonzurgathen2013ModernComputerAlgebra}). Multiplication in $R$ can be implemented using $\bigO(\mathsf{M}(m))$ operations as well.

To test conjugate cancellation, we use the criteria given by \Cref{Thm:Conj-Cancel} or \Cref{Cor:Conj-Cancel-nice-coeffs}. Let $u_1,\dots,u_r\in\FF_q[X]$ be polynomials of degree less than $m$, and let $\alpha$ be the root of a monic irreducible polynomial $f\in\FF_q[X]$ of degree $m$. The goal is to verify $\FF_q(u_1(\alpha),\dots, u_r(\alpha)) = \FF_{q^m}$.
Since $u_s(\alpha)\in\FF_{q^m}$ for $s=1,\dots,r$, we only need to check that $u_s(\alpha)$ does not belong to any subfield of $\FF_{q^m}$ of the form $\FF_{q^d}$, where $d\mid m$. Notice that any such subfield is contained in a subfield $\FF_{q^{m/p}}$, where $p$ is a prime divisor of $m$. Thus, $\FF_q(u_1(\alpha),\dots, u_r(\alpha)) = \FF_{q^m}$ holds if and only if for every prime divisor $p$ of $m$ there exists an $s$ such that $u_s(\alpha)\notin\FF_{q^{m/p}}$ or, equivalently, $u_s(\alpha) \neq u_s(\alpha)^{q^{m/p}}$.

\begin{algorithm}
	\caption{Verify degree of field extension}\label{alg:verify-field-estension}
	\begin{algorithmic}[1]
		\Require $m,r\in\NN$, monic irreducible polynomial $f\in\FF_q[X]$ with $\deg(f) = m$, polynomials $u_1,\dots,u_r\in\FF_q[X]$ with $\deg(u_s) < m$ for all $s = 1,\dots,r$.
		\Ensure $\FF_q(u_1(\alpha),\dots,u_r(\alpha)) = \FF_{q^m}$, where $\alpha\in\FF_{q^m}$ is a root of $f$.

		\State Obtain the set of distinct prime factors $M = \{p_1,\dots,p_k\}$ of $m$.
		\State Let $R = \FF_q[X]/(f)$ and $\xi = X + (f)\in R$. Compute $\xi^{q^{m/p_1}},\dots,\xi^{q^{m/p_k}}$ using the square and multiply algorithm.
		\For{$s = 1,\dots,r$}
			\State Use Horner's scheme to evaluate $u_s(\xi^{q^{m/p}})$ for all $p\in M$.
			\State If $u_s(\xi^{q^{m/p}}) \neq u_s(\xi)$, remove $p$ from $M$.
			\State If $M = \emptyset$ \Return True.
		\EndFor
	\State\Return False
	\end{algorithmic}
\end{algorithm}

\begin{remark}
	\begin{enumerate}
		\item If \Cref{alg:verify-field-estension} is called multiple times on the same polynomial $f\in\FF_q[X]$, then the values $\xi^{q^{m/p_1}},\dots,\xi^{q^{m/p_k}}$ can be reused.
		\item The evaluation of $u_s(\xi)$ comes at no additional cost as $u_s\in\FF_q[X]$ is the canonical representation of $u_s(\xi)$ in $R$.
		\item If $M$ becomes empty in the execution of \Cref{alg:verify-field-estension}, we can prematurely terminate the for-loop.
	\end{enumerate}
\end{remark}

\begin{proposition} \Cref{alg:verify-field-estension} is correct and requires the following number of operations.
	\begin{enumerate}
		\item The cost for the pre-computation in step 2 is $\bigO(m\log(m)\log(q)\mathsf{M}(m))$ operations in $\FF_q$.
		\item Not considering the cost for the pre-computation, the cost for the algorithm is $\bigO(rm\log(m)\mathsf{M}(m))$ operations in $\FF_q$.
		\item Assuming that the inputs $u_1,\dots,u_r$ are independent and identically uniformly distributed, the average case only needs $\bigO(m\log(m)\mathsf{M}(m))$ operations.
	\end{enumerate}
\end{proposition}
\begin{proof}
	The correctness is clear form the discussion preceding \Cref{alg:verify-field-estension}.

	The cost for computing the elements $\xi^{q^{m/p}}$ for a $p\in M$ using square and multiply is bounded by $\bigO(m\log(q)\mathsf{M}(m))$. Since $|M|=k=\bigO(\log(m))$, step 2 requires $\bigO(m\log(m)\log(q)\mathsf{M}(m))$ operations in $\FF_q$.

	The evaluation of $u_s$ at $k = \bigO(\log(m))$ points using Horner's scheme takes $\bigO(m\log(m))$ operations in $R$, and thus $\bigO(m\log(m)\mathsf{M}(m))$ operations in $\FF_q$. The total cost for the for-loop is $\bigO(rm\log(m)\mathsf{M}(m))$.

	Now consider the average case. The estimate $1/2 \leqslant |\FFA_q(m)|/q^m$ can be shown in the same way as \cite[Lemma 2]{rabin1980ProbabilisticAlgorithmsFinite}. As a result, the probability of picking $u_1$ such that $u_1(\xi^{q^{m/p}}) \neq u_1(\xi)$ for all $p\in M$ is $\geqslant 1/2$. Hence, the algorithm terminates on average after two iterations which cost $\bigO(m\log(m)\mathsf{M}(m))$ operations in $\FF_q$.
\end{proof}

The next algorithm uses \Cref{alg:verify-field-estension} to compute an arbitrary number of random $\phi$-products that satisfy conjugate cancellation on the roots of fixed polynomials $f,g$.
We can use those to generate large amounts of irreducible polynomials $f\diamond_{\phi} g$. This last step of computing the composed product can be achieved through the algorithms from \cite{bostan2006FastComputationSpecial},  see also \cite{bostan2002FastComputationTwo,brawley1999computing}. Letting $E = \max\{m,n\}$, it is clear that each iteration of the following \Cref{Alg:compute-many-phi-products} requires $\bigO(E\log(E)\mathsf{M}(E))$ operations in $\FF_q$.

\begin{algorithm}
	\caption{Compute many $\phi$-products satisfying conjugate cancellation for fixed $\alpha,\beta\in\FFA_q(m)\times\FFA_q(n)$}\label{Alg:compute-many-phi-products}
	\begin{algorithmic}[1]
		\Require $m,n,k\in\NN$, polynomials $f,g\in\FF_q[X]$ with $\deg(f) = m$, $\deg(g)=n$.
		\Ensure $\phi_1,\dots,\phi_k\in\FF_q[X,Y]$ that satisfy conjugate cancellation on $(\alpha,\beta)$, where $\alpha$ is a root of $f$ and $\beta$ is a root of $g$.
		\While{$k>0$}
			\State Select $\phi\in\FF_q[X,Y]$ uniformly random with $\deg_X(\phi) < m$ and $\deg_Y(\phi) < n$.
			\State Obtain coefficient polynomials $\psi_0,\dots,\psi_{n-1}$ and $\chi_0,\dots,\chi_{m-1}$ from $\phi$.
			\State Call \Cref{alg:verify-field-estension} with inputs $f$ and $\psi_0,\dots,\psi_{n-1}$.
			\State Call \Cref{alg:verify-field-estension} with inputs $g$ and $\chi_0,\dots,\chi_{m-1}$.
			\If{both calls return True}
				\State $\phi_k = \phi$
				\State $k\gets k-1$
			\EndIf
			\EndWhile
		\State\Return $\phi_1,\dots,\phi_k$.
	\end{algorithmic}
\end{algorithm}

It is also possible to formulate the test for conjugate cancellation using matrices. This test, however, is slower than the one given above due to the relatively high cost of matrix-matrix multiplication. We  mention it here for its illustrative nature.

Assume again that we want to show conjugate cancellation for $(\alpha,\beta)\in\FFA_q(m)\times\FFA_q(m)$ for a $\phi$-product given by
\[
	\phi(X,Y) = \sum_{i=0}^{m-1}\sum_{j=0}^{n-1}c_{ij}X^iY^j = \sum_{i=0}^{m-1} \chi_i(Y)X^i = \sum_{j=0}^{n-1} \psi_j(X)Y^j\in\FF_q[X,Y].
\]
The so-called Petr--Berlekamp matrix $A\in\FF_q^{m\times m}$ (see e.g., \cite[Chapter 14.8]{vonzurgathen2013ModernComputerAlgebra}) is the matrix representation of the Frobenius automorphism in the ordered $\FF_q$-basis $(1,\alpha,\dots,\alpha^{m-1})$ of $\FF_{q^m}$. Let $x_j\in\FF_q^m$ be the representation of $\psi_j(\alpha)$ in that basis. Clearly, \[x_j = (c_{0j}, c_{1j},\dots,c_{m-1,j})^T. \] The equation $\psi_j(\alpha)^{q^d} = \psi_j(\alpha) $ holds if and only if $A^dx_j = x_j$.
Collecting the coefficients of $\phi$ into the matrix $C_\phi = (c_{ij})_{ij}\in\FF_q^{m\times n}$, the $j$-th column of $C_\phi$ is given by $x_j$. Therefore, the equation $(A^d-I)C_\phi = 0$ holds if and only if all coefficient polynomials $\psi_j(X)$ satisfy $\psi_j(\alpha) \in \FF_{q^d}$.
Analogously let $B$ the $n\times n$ Petr--Berlekamp matrix with respect to the ordered basis $(1,\beta,\dots,\beta^{n-1})$. With this, the following proposition is an immediate consequence of \Cref{Cor:Conj-Cancel-nice-coeffs}.

\begin{proposition}
	Let $m,n$ be integers with $\gcd(m,n) = 1$ and $\phi,C_\phi,A,B$ be defined as above. Then the $\phi$-product satisfies conjugate cancellation on the conjugates of $\alpha$ and $\beta$ if and only if
	\begin{enumerate}
		\item $(A^{m/p}-I)C_\phi \neq 0$ for all prime divisors $p$ of $m$,
		\item $(B^{n/p}-I)C_\phi^T \neq 0$ for all prime divisors $p$ of $n$.
	\end{enumerate}
\end{proposition}

\begin{example}
	Let $q = 3$.  We want to construct irreducible polynomials from $f = x^4+x^2+2$ and $g = x^3+2x+1$. This leads to the Petr--Berlekamp matrices
	\[A = \begin{pmatrix}
		1 & 0 & 2 & 0 \\
		0 & 0 & 0 & 2 \\
		0 & 0 & 2 & 0 \\
		0 & 1 & 0 & 0
	\end{pmatrix}
	\quad\text{and}\quad
	B = \begin{pmatrix}
		1 & 2 & 1 \\
		0 & 1 & 1 \\
		0 & 0 & 1
	\end{pmatrix}.\]
	Now we may choose any suitable $\phi$-product. For this example we choose
	\[ \phi(z_1,z_2) = z_1^2+z_2^2+z_1+2z_2. \]
	The corresponding coefficient matrix is given by
	\[ C_\phi = \begin{pmatrix}
		0 & 2 & 1 \\
		1 & 0 & 0 \\
		1 & 0 & 0 \\
		0 & 0 & 0
	\end{pmatrix}.\]
	Since $n=3$ is a prime, there are no subfields except the base field $\FF_3$. We have $(B^1-I)C_\phi^T\neq 0$. For $m = 4$ we need to test for the subfield of size $3^2$. By checking
	\[ (A^2-I)C_\phi = \begin{pmatrix}
		0 & 0 & 0 \\
		1 & 0 & 0 \\
		0 & 0 & 0 \\
		0 & 0 & 0
	\end{pmatrix} \neq 0 \]
	we can confirm that $\phi$ satisfies conjugate cancellation on $\{\alpha,\alpha^3,\alpha^9\}\times\{\beta,\beta^3,\beta^9,\beta^{27}\}$. A direct computation yields the irreducible polynomial
	\[ f \diamond_\phi g = x^{12} + x^{11} + x^{10} + x^8 + 2x^7 + x^6 + x^5 + x^3 + 2x^2 + x + 2. \]
	On the other hand, if we use the product $\zeta(z_1,z_2) = z_1^2 + z_2^2 + 2z_2$ we get $(A^2-I)C_\zeta = 0$, so the product does not satisfy conjugate cancellation. Indeed, we see that
	\[ f \diamond_\zeta g = (x^6 + 2x^5 + x^3 + x^2 + 2x + 1)^2 \]
    is not irreducible.
\end{example}

\subsection{Sufficient criteria for conjugate cancellation}\label{Sec:Construction-Simple}

We consider two special cases in which we can easily verify conjugate cancellation using \Cref{Thm:Conj-Cancel,Cor:Conj-Cancel-nice-coeffs}.

\begin{proposition}
	Let $m,n$ be coprime integers  and let $m_1,n_1$ be the respective smallest prime divisors. Let $\phi(X,Y) = \sum_{i=0}^{m-1}\sum_{j=0}^{n-1}c_{ij}X^iY^j$. If the matrix $C = (c_{ij})_{ij}\in\FF_q^{m\times n}$ has $\rank(C) > \max\{\frac{m}{m_1},\frac{n}{n_1}\}$, then $\phi$ satisfies conjugate cancellation on $\FFA_q(m)\times\FFA_q(n)$.
\end{proposition}
\begin{proof}
	As pointed out in \Cref{Rmk:Existence-Special-Repr}, we can always find linearly independent vectors polynomials $u_1,\dots,u_r\in\FF_q[X]$ and $v_1,\dots,v_r\in\FF_q[Y]$ such that we have $\phi(X,Y) = \sum_{s=1}^ru_s(X)v_s(Y)$. In particular, the number of products in this representation is $r = \rank(C)$.
	Let $\alpha\in\FFA_q(m)$ and $k\mid m$ such that
	\begin{equation*}
		\FF_q(u_1(\alpha),u_2(\alpha),\dots,u_r(\alpha)) = \FF_{q^k}.
	\end{equation*}	As pointed out in the proof of \Cref{Thm:Conj-Cancel}, the $r$ evaluations $u_1(\alpha),u_2(\alpha),\dots,u_r(\alpha)$ are linearly independent over $\FF_q$. Hence, $k = [\FF_{q^k}:\FF_q] \geqslant r > m/m_1$. Since $m/m_1$ is the largest nontrivial divisor of $m$, and $k\mid m$, it follows that $k = m$. An identical argument holds for the $v_i$ under the condition $r>n/n_1$. Conjugate cancellation on $\FFA_q(m)\times\FFA_q(n)$ then follows by \Cref{Thm:Conj-Cancel}.
\end{proof}

A second strategy for constructing diamond products satisfying conjugate cancellation can be found in \cite[Theorem 9]{munemasa2016NoteBrawleyCarlitzTheorem}. We provide a similar criterion, which has fewer restrictions on the coefficient polynomials. The approach is as follows. To apply \Cref{Thm:Conj-Cancel} or \Cref{Cor:Conj-Cancel-nice-coeffs} we can use precisely those polynomials $u$ as coefficient polynomials satisfying $\FF_q(u(\alpha)) = \FF_{q^m}$ for all $\alpha\in \Gap_m\subseteq\FFA_q(m)$.
The following lemma shows that coefficient polynomials with small degree satisfy this property. We give an alternative proof to the original one.

\begin{lemma}\cite[Lemma 5]{munemasa2016NoteBrawleyCarlitzTheorem}
	Let $\psi\in\FF_q[X]$ and $\alpha\in\FFA_q(m)$. Then
	\[ m \leqslant [\FF_q(\psi(\alpha)) : \FF_q]\deg(\psi). \]
\end{lemma}
\begin{proof}
	Let $f$ and $h$ be the minimal polynomials of $\alpha$ and $\psi(\alpha)$ over $\FF_q$ respectively. Because $h(\psi(\alpha)) = 0$ we get $f\mid h\circ \psi$. Therefore, $\deg(f) \leqslant \deg(h)\deg(\psi)$ which yields the desired result.
\end{proof}

This yields the following simple criterion which was also observed in \cite{irimagzi2023DiamondProductsEnsuring,krompholz}.

\begin{proposition}\label{thm:simple-criterion}
	Let $m,n$ be coprime integers and let $m_1,n_1$ be the respective smallest prime divisors. Let \[ \phi(X,Y) = \sum_{i=0}^{m-1} \chi_i(Y)X^i = \sum_{j=0}^{n-1} \psi_j(X)Y^j\in\FF_q[X,Y].\]
	If there exists an index $i$, such that $\deg(\chi_i) < n_1$, and an index $j$, such that $\deg(\psi_j) < m_1$, then $\phi$ satisfies conjugate cancellation on $\FFA_q(m)\times\FFA_q(n)$.
\end{proposition}
\begin{proof}
	From the lemma above it is clear that $\frac{m}{m_1} < [\FF_q(\psi_j(\alpha)):\FF_q]$ for all $\alpha\in\FFA_q(m)$. Since $[\FF_q(\psi_j(\alpha)):\FF_q]$ is a divisor of $m$, it follows that $[\FF_q(\psi_j(\alpha)):\FF_q]=m$, and thus  $\FF_q(\psi_j(\alpha))=\FF_{q^m}$. The claim follows from \Cref{Cor:Conj-Cancel-nice-coeffs}.
\end{proof}

The criterion above yields a simple argument that the diamond product presented in \Cref{Ex:CC-Less-Strict} satisfies conjugate cancellation.

\begin{example}
	We return to \Cref{Ex:CC-Less-Strict} where we constructed a $\phi$-product with $\Ga=\Gb = \FF_{2^6}$ given by
	\[\phi(X,Y) = XY(Y+1).\]
	We have already seen that it does not satisfy weak cancellation. It does, however, satisfy conjugate cancellation on $\Gap_2\times\Gbp_3 = \FFA_2(2)\times\FFA_2(3)$. Letting $u_1(X) = X$ and $v_1(Y) = Y(Y+1)$, it follows easily that $v_1(\beta)\in\FFA_2(3)$ for all $\beta\in\FFA_2(3)$. Applying \Cref{thm:simple-criterion} shows that the $\phi$-product satisfies conjugate cancellation.
\end{example}

\begin{remark}
	It is of interest when an element $\phi(\alpha,\beta)$ is primitive in the extension field, i.e., is a generator of its multiplicative group. Observe that this cannot be the case if the coefficient matrix of $\phi$ has rank 1, in which case we can write $\phi(X,Y) = u_1(X)v_1(Y)$. Indeed,
	since $u_1(\alpha)\in\FF_{q^m}$ and $v_1(\beta)\in\FF_{q^n}$ it follows $k = \ord(u_1(\alpha)) \leqslant  q^m-1$ and $\ell = \ord(v_1(\beta)) \leqslant q^n-1$. This implies
	\[ \ord(u_1(\alpha)v_1(\beta)) = \lcm(k, \ell) \leqslant k\ell = q^{m+n}-q^m-q^n+1 < q^{mn}-1,\]
	where the last inequality holds in the relevant cases $m,n > 1$.
\end{remark}

\section{Diamond products from linearized polynomials}\label{Sec:Normal}
In this section we investigate diamond products arising from polynomials $\phi\in\FF_q[X,Y]$ of the form
\[ \phi(X,Y) = \sum_{i=0}^{m-1}\sum_{j=0}^{n-1}c_{ij}X^{q^i}Y^{q^j}, \]
i.e., $\phi$ is bilinear over $\FF_q$.
We can apply the same methods, i.e., \Cref{Lma:Coeff-Polys,Lma:Main-Lemma} for this case as well.
The difference in the approach here is that instead of polynomial bases of the form $\{1,\alpha,\dots,\alpha^{m-1}\}$, we need to use \emph{normal bases}. An element $\alpha\in\FF_{q^m}$ is called normal over $\FF_q$, if the conjugates $\alpha,\alpha^q,\dots,\alpha^{q^{m-1}}$ are linearly independent over $\FF_q$. Define
\[ \mathcal{N}_q(m) := \{\alpha\in\FF_{q^m}:\alpha\text{ is normal over } \FF_q\}. \]
It is well known that $\mathcal{N}_q(m) \neq \emptyset$ for any choice of $q,m$.
Obviously, $\mathcal{N}_q(m)$ is a Frobenius invariant set, and the inclusion $\mathcal{N}_q(m) \subseteq \FFA_q(m)$ holds. Additionally, recall that we call $u\in\FF_q[X]$ a $q$-polynomial if it has the form $u(X) = \sum_i c_i X^{q^i}$.

\subsection{Equivalent and sufficient criteria for conjugate cancellation}
We give an analog of \Cref{Thm:Conj-Cancel} for the case of bilinear $\phi$-products. Note that compared to \Cref{Thm:Conj-Cancel}, we can choose much larger degrees. The proof of the theorem is essentially identical to \Cref{Thm:Conj-Cancel}, with polynomial bases exchanged for normal bases.
\begin{theorem}\label{Lma:Coeff-Polys-Linear}
	Let $m,n > 1$ be coprime and let $\phi\in\FF_q[X,Y]$ be bilinear over $\FF_q$ with a representation
	\[ \phi(X,Y) = \sum_{s=1}^r u_s(X)v_s(Y), \]
	where the $q$-polynomials $u_s\in\FF_q[X],~v_s\in\FF_q[Y]$ are linearly independent over $\FF_q$ with $\deg(u_s) < q^m$ and $\deg(v_s)<q^n$ for all $s$.
	The following are equivalent
	\begin{enumerate}
		\item The $\phi$-product satisfies conjugate cancellation on $\mathcal N_q(m)\times \mathcal N_q(n)$. \label{itm:41a}
		\item There exist $(\alpha,\beta)\in\mathcal{N}_q(m)\times\mathcal{N}_q(n)$ with $\FF_q(u_1(\alpha),\dots,u_r(\alpha)) = \FF_{q^{m}}$ and $\FF_q(v_1(\beta),\dots,v_r(\beta)) = \FF_{q^{n}}$. \label{itm:41b}
		\item For all $(\alpha,\beta)\in\mathcal{N}_q(m)\times\mathcal{N}_q(n)$ we have $\FF_q(u_1(\alpha),\dots,u_r(\alpha)) = \FF_{q^{m}}$ and $\FF_q(v_1(\beta),\dots,v_r(\beta)) = \FF_{q^{n}}$. \label{itm:41c}
	\end{enumerate}
\end{theorem}
\begin{proof}
	Let $\beta\in\mathcal{N}_q(n)$. Since $\gcd(m,n) = 1$, it follows from \cite[Theorem 2.3.2]{gao1993NormalBasesFinite} that the elements $\{\beta,\beta^q,\beta^{q^2},\dots,\beta^{q^{n-1}}\}$ also form a normal basis of $\FF_{(q^m)^n}$ over $\FF_{q^m}$. Hence, $\beta\in\mathcal{N}_{q^m}(n)$ and the inclusion $\mathcal{N}_q(n)\subseteq\mathcal{N}_{q^m}(n)$ holds.

    Let $c_s\in\FF_{q^m}$ such that
	\[0 = \sum_{s=1}^r c_sv_s(\beta) = \left(\sum_{s=1}^r c_sv_s\right)(\beta).\]
	This equation describes a linear combination of $\beta,\beta^q,\dots,\beta^{q^{n-1}}$. It follows from $\beta\in\mathcal{N}_{q^m}(n)$ that $\sum_{s=1}^r c_sv_s = 0$. Using the same argument as in the proof of \Cref{Thm:Conj-Cancel}, the linear independence of $v_1,\dots,v_r$ over $\FF_q$ yields $c_1,\dots,c_r = 0$.
	The equivalence of \ref*{itm:41a} and \ref*{itm:41b} now follows by an identical argument on $u_1,\dots,u_r$ and the application of \Cref{Lma:Main-Lemma}.

	The implication \ref*{itm:41c} $\implies$ \ref*{itm:41b} is obvious. Thus, it remains to show \ref*{itm:41b} $\implies$ \ref*{itm:41c}. Let $\gamma\in\mathcal{N}_q(m)$ be an arbitrary normal element. Let $c_0,\dots,c_{n-1}\in\FF_q$ such that $\alpha = c_0\gamma+c_1\gamma^q+\dots+c_{n-1}\gamma^{q^{n-1}}$. Then
	\begin{align*}
		u_1(\alpha) &= u_1(c_0\gamma+c_1\gamma^q+\dots+c_{n-1}\gamma^{q^{n-1}}) \\
		&= c_0u_1(\gamma)+c_1u_1(\gamma)^q+\dots+c_{n-1}u_1(\gamma)^{q^{n-1}} \in\FF_q(u_1(\gamma)).
	\end{align*}
	This shows $\FF_q(u_1(\alpha)) \subseteq \FF_q(u_1(\gamma))$. Repeating this argument for the remaining $u_2,\dots,u_r$ shows
	\[\FF_{q^m} = \FF_q(u_1(\alpha),u_2(\alpha),\dots,u_r(\alpha)) \subseteq \FF_q(u_1(\gamma),u_2(\gamma),\dots,u_r(\gamma)),\]
	and thus $\FF_q(u_1(\gamma),u_2(\gamma),\dots,u_r(\gamma)) = \FF_{q^m}$. Since $\gamma$ was arbitrary, this implies (3).
\end{proof}

\Cref{Lma:Coeff-Polys-Linear} gives an analogous criterion to \Cref{Thm:Conj-Cancel} in the case of linearized coefficient polynomials, that is, bilinear diamond products. It is clear that a similar statement to \Cref{Cor:Conj-Cancel-nice-coeffs} can be shown in the case of bilinear $\phi$-products by replacing $X^i$ with $X^{q^i}$ and $Y^j$ with $Y^{q^j}$.

We remark that, analogous to \Cref{Sec:Algorithm}, we can characterize conjugate cancellation for bilinear diamond products using matrices. In this case however, we can make additional use of the fact that we are working with normal elements. Let $\alpha\in\mathcal N_q(m)$. The matrix representation of the Frobenius automorphism in the normal basis $\alpha,\alpha^q,\dots,\alpha^{q^{m-1}}$ is given by
\[A = \begin{pmatrix}
	0 & \cdots & 0 & 1 \\
	1 & \ddots &   & 0 \\
	  & \ddots & \ddots & \vdots \\
	  &   & 1 & 0 &   \\
\end{pmatrix}\in\FF_q^{m\times m}.\]
This matrix simply describes the cyclic shift of a vector. We can now apply the same approach as in \Cref{Sec:Algorithm} by letting $C = (c_{ij})_{ij}\in\FF_q^{m\times n}$ be the coefficient matrix of the $\phi$-product
\[\phi(X,Y) = \sum_{i=0}^{m-1}\sum_{j=0}^{n-1}c_{ij}X^{q^i}Y^{q^j} = \sum_{j=0}^{n-1} \psi_j(X)Y^{q^j}.\]
Then $x_j = (c_{0j},c_{1j},\dots,c_{m-1,j})^T$ is again the representation of $\psi_j(\alpha)$ in the normal basis generated by $\alpha$, and is also the $j$-th column of $C$. This means that the condition \ref*{itm:41a} in \Cref{Lma:Coeff-Polys-Linear} can be verified by checking $(A^d-I)C\neq 0$ for all proper divisors $d$ of $m$.
Note that this test only uses comparisons of elements and no arithmetic operations in $\FF_q$.

Analogous to the methods presented in \Cref{Sec:Construction-Simple}, we now provide a simple criterion to test if a linearized polynomial over $\FF_q$ satisfies $\psi(\alpha)\in\FFA_q(m)$.

\begin{theorem} \label{thm:linearized}
	Let $m>1$ and $m_1$ be the largest divisor of $m$ such that $m_1<m$. Let further $\psi\in\FF_q[X]$ be a $q$-polynomial with $\deg(\psi) < q^{m-m_1}$. Then $\psi(\alpha)\in\mathcal{F}_q(m)$ for all $\alpha\in\mathcal{N}_q(m)$.
\end{theorem}

\begin{proof}
	Let $\alpha\in\mathcal{N}_q(m)$ be fixed and let the integer $k$ be a divisor of $m$ such that $\FF_q(\psi(\alpha)) = \FF_{q^k}$. Since $\alpha$ is normal and $\psi$ is $q$-linear, it easily follows that $\psi(x)\in\FF_{q^k}$ for all $x\in\FF_{q^m}$. Notice that $\dim(\img(\psi)) \leqslant k$ if we view $\psi$ as a linear mapping over $\FF_q$. The rank-nullity theorem then implies $\dim(\ker(\psi)) \geqslant m-k$, so $\deg(\psi) \geqslant q^{m-k}$. The hypothesis $\deg(\psi) < q^{m-m_1}$ then implies $k = m$ and $\psi(\alpha)\in\mathcal{F}_q(m)$.
\end{proof}
Theorem~\ref{thm:linearized} thus yields a simple criterion when a $q$-polynomial defines a $\phi$-product that satisfies conjugate conjugation via Theorem~\ref{Lma:Coeff-Polys-Linear}.

\subsection{When is \texorpdfstring{$\phi(\alpha,\beta)$}{ɸ(α,β)} a normal element?}
A challenge in the study of normal elements is to describe transformations preserving normality.
In this subsection we study for which choices of $\phi$ the image $\phi(\alpha,\beta)$ of two normal elements is again normal.
It is interesting to note that this is never the case for the
sum of two normal elements as \cref{prop:normal_sum} shows, while the multiplication of two
normal elements is always normal \cite[Theorem 2.3.3]{gao1993NormalBasesFinite}.
The diamond product
 $\alpha\diamond\beta=\alpha+\beta+d$ with $d\in\FF_q$ generalizes the field addition and  was introduced in \cite{brawley1987IrreduciblesComposedProduct}.

\begin{proposition}\label{prop:normal_sum}
	Let $(\alpha,\beta)\in\mathcal{N}_q(m)\times\mathcal{N}_q(n)$ where $m,n>1$ and $\gcd(m,n)=1$. Then $\alpha+\beta+d$ is not normal over $\FF_q$ for all $d\in\FF_q$.
\end{proposition}
\begin{proof}
	Starting with $d = 0$, let $s_k\in\FF_q$ for $k = 0,\dots,mn-1$ such that
	\[0 = \sum_{k=0}^{mn-1} s_k (\alpha+\beta)^{q^k}
	= \sum_{i=0}^{m-1}\sum_{j=0}^{n-1} s_{jm + i} \alpha^{q^i} + \sum_{j=0}^{n-1}\sum_{i=0}^{m-1} s_{in + j} \beta^{q^j}.\]
	In particular, the above equation is satisfied if
	\[ 0 = \sum_{i=0}^{m-1}\sum_{j=0}^{n-1} s_{jm + i}\alpha^{q^i} \quad\text{ and }\quad 0 = \sum_{j=0}^{n-1}\sum_{i=0}^{m-1} s_{in + j} \beta^{q^j}. \]
	This leads to the $m+n$ equations $\sum_{j=0}^{n-1} s_{jm + i} = 0$ for all $i$ and $\sum_{i=0}^{m-1} s_{in + j} = 0$ for all $j$. Since we have $mn$ unknowns $s_k$, there will always be a non-trivial solution.

	Let $d\in\FF_q$ be arbitrary. Using the same $s_k$ as above we have
	\begin{align*}
		\sum_{k=0}^{mn-1}s_k(\alpha+\beta+d)^{q^k} &= d\sum_{k=0}^{mn-1}s_k + \sum_{k=0}^{mn-1}s_k(\alpha+\beta)^{q^k} \\
		&= d\sum_{i=1}^{m-1}\sum_{j=0}^{n-1} s_{jm + i} + 0\\
		&= 0.
	\end{align*}
\end{proof}

We provide a simple criterion for the normality of the evaluations of $\phi$-products given by
\begin{equation}\label{Equ:Staircase-Def-Repr}
	\phi(X,Y) = \sum_{i=0}^{m-1}\sum_{j=0}^{n-1}c_{ij}X^{q^i}Y^{q^j}.
\end{equation}
For this representation we introduce the so-called staircase-polynomials.

\begin{definition}
	Given a $q$-linear $\phi$-product as in \eqref{Equ:Staircase-Def-Repr}, the staircase-polynomial $e\in\FF_q[X]$ is defined as
	\[ e=\sum_{k = 1}^{mn-1} c_{k\text{ mod } m, k\text{ mod } n} X^k. \]
\end{definition}

The name is motivated from the idea that the coefficients of $\phi$ can be arranged in a matrix $C = (c_{ij})_{ij}$ of size $m\times n$. Then the coefficients from the staircase-polynomial are obtained by following the diagonal.

Let $(\alpha,\beta)\in\mathcal{N}_q(m)\times\mathcal{N}_q(n)$, where $\gcd(m,n) = 1$. Then it can be shown that the product $\alpha\beta$ is a normal element \cite[Theorem 2.3.3]{gao1993NormalBasesFinite}. Hence, every element in $\FF_{q^{mn}}$ can be represented using the ordered $\FF_q$-basis $(\alpha\beta, (\alpha\beta)^q,\dots,(\alpha\beta)^{q^{mn-1}})$. By \cite[Theorem 2.2.6]{gao1993NormalBasesFinite} 
a linear combination 
$\gamma = \sum_{k=0}^{mn-1} r_k (\alpha\beta)^{q^k}$, where $r_k\in\FF_q$, is normal if and only if the polynomial $\gamma(X) = \sum_{k=0}^{mn-1} r_kX^k$ is coprime to $X^{mn}-1$. Observe that $\phi(\alpha,\beta)$ can be represented in the ordered basis induced by $\alpha\beta$. Thus, choosing $r_k$ such that $\gamma = \phi(\alpha,\beta)$ we have
\[ \sum_{i=0}^{m-1}\sum_{j=0}^{n-1}c_{ij}\alpha^{q^i}\beta^{q^j} = \sum_{k=0}^{mn-1} r_k (\alpha\beta)^{q^k} = \sum_{k=0}^{mn-1} r_k \alpha^{q^{k\text{ mod } m}}\beta^{q^{k\text{ mod } n}},\]
where $c_{ij}$ is defined by \eqref{Equ:Staircase-Def-Repr}. Therefore, $r_k = c_{k\text{ mod } m, k\text{ mod } n}$. This means that $\gamma(X)$ is exactly the staircase-polynomial of $\phi$. Hence, $\phi(\alpha,\beta)$ is normal if and only if the staircase-polynomial is coprime to $X^{mn}-1$.

Hence, Theorem 2.2.6 from \cite{gao1993NormalBasesFinite} reduces in our notation to:

\begin{theorem}\label{Thm:Staircase-Normal}
	Let $(\alpha,\beta)\in\mathcal{N}_q(m)\times \mathcal{N}_q(n)$, where $\gcd(m,n) = 1$. Then $\phi(\alpha,\beta)$ is normal if and only if the staircase-polynomial is relatively prime to $X^{mn}-1$.
\end{theorem}

We now consider the special case of $\phi(\alpha,\beta)=\alpha^{q^k}\beta\pm \alpha\beta^{q^\ell}$.
These $\phi$-products satisfy conjugate cancellation for normal elements by \Cref{Lma:Coeff-Polys-Linear} as we can express $\phi$ by $\phi(X,Y) = u_1(X)v_1(Y) + u_2(X)v_2(Y)$ with $u_2(X) = X$ and $v_1(Y) = Y$.
Note that for certain choices of $q,k,l$ these $\phi$-products are actually the multiplication of a semifield, called the generalized twisted fields~\cite{albert1961generalized}.

\begin{theorem}\label{thm:albert}
	Let $(\alpha,\beta)\in\mathcal{N}_q(m)\times\mathcal{N}_q(n)$ where $\gcd(m,n)=1$. Then $\alpha^{q^k}\beta+\alpha\beta^{q^\ell}$ is normal over $\FF_q$ if and only if $q$ is odd and one of the properties holds:
	\begin{enumerate}
		\item $m,n$ are odd
		\item $m$ is even with $\nu_2(m) \leqslant \nu_2(k)$
		\item $n$ is even with $\nu_2(n) \leqslant \nu_2(\ell)$
	\end{enumerate}
\end{theorem}
\begin{proof}
	The staircase polynomial of $X^{q^k}Y+XY^{q^\ell}$ is given by $X^s+X^t$ for some integers $s,t$. In particular, $s$ and $t$ are uniquely determined by the Chinese remainder theorem as
	\begin{align*}
		s &= kn\overline{n} \\
		t &= \ell m \overline{m}
	\end{align*}
    where $\overline{n}$ is the least positive residue of the inverse of $n$ modulo $m$; and similarly $\overline{m}$ is the least positive residue of the inverse of $m$ modulo $n$.

    We immediately see that if $q$ is even then 1 is a root of both $X^{mn}-1$ and $X^s+X^t$, and Theorem~\ref{Thm:Staircase-Normal} gives the result.

    So assume that $q$ is odd.
	To calculate $\gcd(X^{mn}-1, X^s+X^t)$ we may assume $s < t$. The case $s>t$ can be done in the same way. Because $X^{mn}-1$ has no root in $0$, we have $\gcd(X^{mn}-1, X^s+X^t) = \gcd(X^{mn}-1, X^{t-s}+1)$. We may now use this well known property
	\[ \gcd(X^i-1, X^r+1) = \left\{\begin{array}{cl}
		X^{\gcd(i,r)}+1 & \text{if } \frac{i}{\gcd(i,r)} \text{ is even} \\
		1 & \text{else.}
	\end{array}\right.\]
	We now go over the different possible cases that may arise.
	\begin{enumerate}
		\item If $m$ and $n$ are odd, it follows that $X^{mn}-1$ and $X^s+X^t$ are coprime.
		\item If $m$ is even, then so is $t$. Because $m,n$ are coprime it follows that $n$ and $\overline{n}$ are odd. In particular, $\nu_2(s) = \nu_2(k)$. Therefore, $\nu_2(t-s) = \min\{\nu_2(t), \nu_2(k)\}$ and $\nu_2(t) \geqslant \nu_2(m)$. Observe that
		\begin{equation}\label{Equ:Thm47-1}
			\begin{split}
				\nu_2\left(\frac{mn}{\gcd(mn,t-s)}\right) &= \nu_2(m) - \min\{\nu_2(m),\nu_2(t-s)\} \\
				&= \nu_2(m) - \min\{\nu_2(m),\nu_2(t),\nu_2(k)\} \\
				&= \nu_2(m) - \min\{\nu_2(m),\nu_2(k)\}.
			\end{split}
		\end{equation}
		From this it follows that \eqref{Equ:Thm47-1} becomes $0$ if and only if $\nu_2(m) \leqslant \nu_2(k)$.
		\item The case $n$ even is done in the same way as (ii). In particular, $\nu_2(n) \leqslant \nu_2(\ell)$ is equivalent to $\nu_2\left(\frac{mn}{\gcd(mn,t-s)}\right) = 0$.
	\end{enumerate}
	Finally, the application of \Cref{Thm:Staircase-Normal} yields the desired result.
\end{proof}

\begin{proposition}\label{prop:albert}
	Let $(\alpha,\beta)\in\mathcal{N}_q(m)\times\mathcal{N}_q(n)$ where $\gcd(m,n)=1$. Then $\alpha^{q^k}\beta-\alpha\beta^{q^\ell}$ is never normal over $\FF_q$.
\end{proposition}
\begin{proof}
	As in the previous proof we may use that the staircase polynomial of the diamond product is given by $X^s-X^t$ for some integers $s,t$. Observe that both the staircase polynomial and $X^{mn}-1$ have a root at $1$.
	Thus, the staircase polynomial always has a common divisor with $X^{mn}-1$.
\end{proof}

\begin{remark}
    As mentioned before, \cref{thm:albert} and \cref{prop:albert} in particular deal with the cases that the $\phi$-product is the multiplication of a semifield, which means that the binary operation $x \circ y :=\phi(x,y)$ distributes over addition and has no zero divisors; in our cases the specific semifield operation was that of a \emph{generalized twisted field}, see~\cite{albert1961generalized}. It is a natural question whether similar results are possible for other semifield multiplications. The general case is however not as straightforward: By \cref{Thm:Staircase-Normal}, we need to check if the staircase polynomial of the semifield multiplication is relatively prime to $X^{mn}-1$. For the generalized twisted fields we considered, this is comparatively easy since the staircase polynomial is a binomial and has thus a very simple structure. In the general case, the staircase polynomial of a semifield will be much more complicated. Indeed, the complexity of the staircase polynomial of a semifield multiplication is closely related to the matrix rank and BEL-rank of the semifield, for definitions see~\cite[Section 4]{lavrauw2017bel}. In~\cite{lavrauw2017bel} it is shown that the twisted fields have BEL-rank 2 (leading to a binomial for a $\phi$-product), and computational results in~\cite[Section 7]{lavrauw2017bel}  indicate that most other semifields have significantly higher matrix rank and BEL-rank, yielding more complex staircase polynomials that are harder to analyze.
\end{remark}

\bibliographystyle{plain}
\bibliography{ref_zotero,ref_manual}

\end{document}